\documentclass{amsart}
\makeatletter
\@namedef{subjclassname@2020}{%
  \textup{2020} Mathematics Subject Classification}
\makeatother % Cambiar el año de la MSC, de 2010 a 2020

\usepackage{amsthm,amssymb,amsfonts,latexsym,mathtools,thmtools, mathrsfs}
\usepackage[T1]{fontenc}
\usepackage{mathrsfs}
\usepackage{tikz-cd} % Commutative diagrams.
\usepackage{enumitem} % Customization of items.
\usepackage{hyperref} %hyperlinks
\usepackage[all]{xy}
\hypersetup{
    colorlinks=true,
    linkcolor=blue,
    filecolor=blue,      
    urlcolor=blue,
    linktocpage=true
}

\newtheorem{theorem}{Theorem}[section]
\newtheorem{lemma}[theorem]{Lemma}

\theoremstyle{definition}
\newtheorem{definition}[theorem]{Definition}

\theoremstyle{remark}
\newtheorem{remark}[theorem]{Remark}

\newtheorem{proposition}[theorem]{Proposition}
\newtheorem{corollary}[theorem]{Corollary}

\numberwithin{equation}{section}

\begin{document}

\title[Maps between schematic semi-graded rings]{Maps between schematic semi-graded rings}

%    Remove any unused author tags.

%    author one information

\author{Andr\'es Chac\'on}
\address{Universidad Nacional de Colombia - Sede Bogot\'a}
\curraddr{Campus Universitario}
\email{anchaconca@unal.edu.co}
\thanks{}

%   author two information

\author{Mar\'ia Camila Ram\'irez}
\address{Universidad Nacional de Colombia - Sede Bogot\'a}
\curraddr{Campus Universitario}
\email{macramirezcu@unal.edu.co}
\thanks{}

%   author three information

\author{Armando Reyes}
\address{Universidad Nacional de Colombia - Sede Bogot\'a}
\curraddr{Campus Universitario}
\email{mareyesv@unal.edu.co}

\thanks{The authors were supported by the research fund of Faculty of Science, Code HERMES 53880, Universidad Nacional de Colombia - Sede Bogot\'a, Colombia.}



\date{}

\dedicatory{Dedicated to Professor Oswaldo Lezama}

\begin{abstract}

Motivated by Smith's work \cite{Smith2003, Smith2016} on maps between non-commu\-tative projective spaces of the form ${\rm Proj}_{nc} A$ in the setting of non-commutative projective geometry developed by Rosenberg and Van den Bergh, and the notion of schematicness introduced by Van Oystaeyen and Willaert \cite{VanOystaeyenWillaert1995} to $\mathbb{N}$-graded rings with the aim of formulating a non-commutative scheme theory \`a la Grothendieck \cite{EGAII1961}, in this paper we consider a first approach to maps in the Smith's sense in the more general setting of non-commutative projective spaces over semi-graded rings defined by Lezama and Latorre \cite{LezamaLatorre2017}. We extend Smith's key result \cite[Theorem 3.2]{Smith2003}, \cite[Theorem 1.2]{Smith2016} from the category of schematic $\mathbb{N}$-graded rings to the category of schematic semi-graded rings. 

\end{abstract}

\maketitle

\section{Introduction}

It is well-known that Serre \cite{Serre1955} proved a theorem that describes the coherent sheaves on a projective scheme in terms of graded modules. Following Artin and Zhang \cite[Section 1]{ArtinZhang1994}, a commutative graded $\Bbbk$-algebra is associated to a projective scheme ${\rm Proj} A$, and the geometry of this scheme can be described in terms of the quotient category $\mathsf{qgr} A = \mathsf{gr} A / \mathsf{tors}$, where $\mathsf{gr} A$ denotes the category of graded modules and $\mathsf{tors}$ denotes its subcategory of torsion modules. For $A$ a finitely generated commutative graded $\Bbbk$-algebra and $X$ its associated projective scheme, if $\mathsf{coh} X$ denotes the category of coherent sheaves on $X$ and $\mathcal{O}_X(n)$ is the $n$th power of the twisting sheaf on $X$ \cite[p. 117]{Hartshorne1977}, then we have a functor $\Gamma_*: \mathsf{coh} X \to \mathsf{qgr}\ A$ given by
\[
\Gamma_*(\mathcal{F}) = \bigoplus_{d = -\infty}^{\infty} {\rm H}^0(X, \mathcal{F}\otimes \mathcal{O}_X(d)),
\]

and {\em Serre's theorem} asserts that if $A$ is generated over $\Bbbk$ by elements of degree one, then $\Gamma_*$ defines an equivalence of categories $\mathsf{coh} X \to \mathsf{qgr}\ A$ \cite[Section 59, Proposition. 7.8, p. 252]{Serre1955}, \cite[3.3.5]{EGAII1961} and \cite[Proposition. II. 5.15]{Hartshorne1977}.

Artin and Zhang \cite{ArtinZhang1994} extended Serre's theorem. For $A$ an $\mathbb{N}$-graded algebra over a commutative Noetherian ring, they defined the associated projective scheme to be the pair ${\rm Proj}\ A = ({\rm qgr}\ A, \mathcal{A})$, where ${\rm qgr}\ A$ is the quotient category above and $\mathcal{A}$ is the object determined by the right module $A_A$. They showed the autoequivalence $s$ of ${\rm qgr}\ A$ defined by the shift of degrees in ${\rm gr}\ A$. The object $\mathcal{A}$ plays the role of the {\em structure sheaf} of ${\rm Proj}\ A$ and $s$ the role of the {\em polarization} defined by the projective embedding (this definition is the same as is given by Verevkin \cite{Verevkin1992a, Verevkin1992}). Since Serre's theorem does not hold for all commutative graded algebras, i.e., the functor defined by $\Gamma_*$ need not be an equivalence, Artin and Zhang's definition of ${\rm Proj}\ A$ is compatible with the classical definition for commutative graded rings only under some additional hypotheses, such as that $A$ is generated in degree one. In the literature, the non-commutative version of Serre's theorem is known as {\em Serre-Artin-Zhang-Verevkin theorem} \cite{ArtinZhang1994, Verevkin1992a, Verevkin1992}. 

Van Oystaeyen and Willaert \cite{VanOystaeyenWillaert1995} studied above ${\rm Proj}\ A$ by developing a kind of scheme theory similar to the commutative theory (note that Manin \cite{Manin1991} commented the failure of attempts to obtain a non-commutative scheme theory \`a la Grothendieck \cite{EGAII1961} for quantized algebras). They noticed that this theory is possible only if the connected and $\mathbb{N}$-graded algebra considered contains \textquotedblleft enough\textquotedblright\ Ore sets. Algebras satisfying this condition are called {\em schematic}, and some examples are homogenizations of almost commutative algebras, Rees rings of universal enveloping algebras of Lie algebras, and three-dimensional Sklyanin algebras. They constructed a {\em generalized Grothendieck topology} for the free monoid on all Ore sets of a schematic algebra $R$, and defined a {\em non-commutative site} (see \cite{VanOystaeyenWillaert1996a} for more details) as a category with coverings on which sheaves can be defined, and formulated the Serre's theorem. As a consequence of their treatment, an equivalence between the category of all coherent sheaves and the category ${\rm Proj}\ A$ was obtained in the sense of non-commutative algebraic geometry introduced by Artin \cite{Artin1992}. In a series of papers \cite{VanOystaeyenWillaert1996a, vanOystaeyenWillaert1996, VanOystaeyenWillaert1997, Willaert1998}, and the book \cite{VanOystaeyen2000}, Van Oystaeyen and Willaert investigated several properties of schematic $\mathbb{N}$-graded rings.

On the other hand, following Rosenberg \cite{Rosenberg1995, Rosenberg1998} and Van den Bergh's \cite{VandenBergh2001} ideas of the non-commutative projective geometry, Smith \cite{Smith2003, Smith2016} takes a {\em Grothendieck category} as the basic non-commutative geometric object. More exactly, we think of a Grothendieck category $\mathsf{Mod} X$ (the idea for the notation $X = \mathsf{Mod} X$ is Van den Bergh's) as \textquotedblleft the quasi-coherent sheaves on an imaginary non-commutative space $X$\textquotedblright. The standard commutative example is the category $\mathsf{Qcoh} X$ of quasi-coherent sheaves on a quasi-separated, quasi-compact scheme $X$. The two non-commutative models are $\mathsf{Mod} R$, the category of right modules over a ring, and $\mathsf{Proj} A$, the non-commutative projective spaces having a (not necessarily commutative) graded ring $A$ as homogeneous coordinate ring, which were defined by Verevkin \cite{Verevkin1992, Verevkin1992a} and Artin and Zhang \cite{ArtinZhang1994} (for more details, see \cite[Definition 2.2]{Smith2002}). A $\mathsf{map}$ $g:Y \to X$ between two spaces is an adjoint pair of functors $(g^{*}, g_{*})$ with $g_{*}: \mathsf{Mod} Y \to \mathsf{Mod} X$ and $g^{*}$ left adjoint to $g_{*}$. If $g_{*}$ is faithful and has a right adjoint, then $g$ is called $\mathsf{affine}$. An immediate example is a ring homomorphism $\varphi: R\to S$ that induces an affine map $g: Y \to X$ between the affine spaces defined by $\mathsf{Mod} Y := \mathsf{Mod} S$ and $\mathsf{Mod} X := \mathsf{Mod} R$. In his papers, Smith considered the question on maps between non-commutative projective spaces. He called a map $g:Y \to X$ a $\mathsf{closed\ immersion}$ if it is affine and the essential image of $\mathsf{Mod} Y$ in $\mathsf{Mod} X$ under $g_{*}$ is closed under submodules and quotients. For $J$ a graded ideal in a not necessarily commutative $\mathbb{N}$-graded $\Bbbk$-algebra $A = A_0 \oplus A_1 \oplus \dotsb$ in which ${\rm dim}_{\Bbbk} A_i < \infty$ for all $i$, Smith \cite[Theorem 3.2]{Smith2003}, \cite[Theorem 1.2]{Smith2016} showed that a surjective homomorphism $A \to A/J$ of graded rings induces a closed immersion $i: {\rm Proj}_{nc} A/J \to {\rm Proj}_{nc} A$ between the non-commutative projective spaces with homogeneous coordinate rings $A$ and $A/J$.

With the aim of generalizing $\mathbb{N}$-graded rings, finitely $\mathbb{N}$-graded algebras and other families of algebras appearing in ring theory and non-commutative geometry that are not $\mathbb{N}$-graded algebras (of course in a non-trivial sense), Lezama and Latorre \cite{LezamaLatorre2017} introduced the {\em semi-graded rings} generated in degree one. In their paper, they investigated geometric and algebraic properties of these objects such as generalized Hilbert series, Hilbert polynomial and Gelfand-Kirillov dimension. In particular, they extended the notion of non-commutative projective scheme in the sense of Artin and Zhang above for $\mathbb{N}$-graded rings to the setting of semi-graded rings, and generalized the Serre-Artin-Zhang-Verevkin theorem \cite[Section 6]{LezamaLatorre2017}. Since then, algebraic and homological properties of semi-graded rings have been studied by different researchers \cite{AbdiTalebi2024, Artamonov2015, Fajardoetal2020, HigueraReyes2023, Lezama2021,  Lezamaetal2019, NinoRamirezReyes2020, NinoReyes2023}.

Recently, the first and third authors \cite{ChaconPhD2022, ChaconReyes2022} investigated the {\em schematicness} of semi-graded rings and extended Van Oystaeyen and Willaert's ideas \cite{VanOystaeyenWillaert1995} on a scheme theory for non-commutative $\mathbb{N}$-graded rings to the class of semi-graded rings (as a matter of fact, the condition of connectedness of the algebra appearing in \cite{VanOystaeyenWillaert1995} is not assumed), and presented another approach to the Serre-Artin-Zhang-Verevkin theorem for semi-graded rings \cite[Section 5]{ChaconReyes2022}.

With all above facts in mind, in this paper we consider maps in the Smith's sense but now in the setting of non-commutative projective spaces over semi-graded rings. We extend Smith's key result \cite[Theorem 3.2]{Smith2003}, \cite[Theorem 1.2]{Smith2016} from the category of schematic $\mathbb{N}$-graded rings to the category of schematic semi-graded rings. This paper contributes to the research on noncommutative geometry for semi-graded rings developed in the literature (e.g. \cite{Fajardoetal2020, HernandezReyes2020, Lezama2020, Lezama2021, LezamaLatorre2017, LezamaGomez2019, ReyesSarmiento2022, SuarezReyesSuarez2023}, and references therein).

The article is organized as follows. In Section \ref{definitionsandpreliminaries}, we recall some definitions and preliminaries on schematic semi-graded rings that are necessary for the rest of the paper. In Section \ref{semigradedsettingnew} we present the original results of the paper. Our key results are Propositions \ref{Adjointparfirst} and \ref{Adjointparsecond}, and Theorem \ref{Smith2003Theorem3.2Generalization}. Finally, Section \ref{conclusionsfuturework} contains some ideas for future work.

Throughout the paper, the term ring means an associative ring with identity not necessarily commutative. The symbols $\mathbb{N}$ and $\mathbb{Z}$ denote the set of natural numbers including zero and the ring of integer numbers, respectively. The term module will always mean left module unless stated otherwise. The category of left $R$-modules is written as $\mathsf{Mod} R$.

\section{Definitions and preliminaries}\label{definitionsandpreliminaries}

Lezama and Latorre \cite{LezamaLatorre2017} presented an introduction to the non-commutative algebraic geometry for non-$\mathbb{N}$-graded algebras and finitely non-graded algebras by defining a new class of rings, the {\em semi-graded rings}. These rings extend several kinds of non-commutative rings of polynomial type such as Ore extensions \cite{Ore1931, Ore1933}, families of differential operators generalizing Weyl algebras and universal enveloping algebras of finite dimensional Lie algebras such as PBW extensions \cite{BellGoodearl1988}, algebras appearing in mathematical physics \cite{IPR01, ReyesSarmiento2022}, ambiskew polynomial rings \cite{Jordan2000}, 3-dimensional skew polynomial rings \cite{BellSmith1990, ReyesSarmiento2022, Rosenberg1995}, bi-quadratic algebras on 3 generators with PBW bases \cite{Bavula2023}, skew PBW extensions \cite{GallegoLezama2011, LezamaReyes2014}, and other algebras having PBW bases \cite{GolovashkinMaksimov1998, GolovashkinMaksimov2005, NinoReyes2023}. A detailed list of examples of semi-graded rings and its relationships with other algebras can be found in Fajardo et al. \cite{Fajardoetal2020}. 

\begin{definition}[{\cite[Definition 2.1]{LezamaLatorre2017}}]\label{def.SG2}
Let $R$ be a ring. $R$ is said to be {\em semi-graded} ($\mathsf{SG}$) if there exists a collection $\{R_n\}_{n\in\mathbb{Z}}$ of subgroups $R_n$ of the additive group $R^{+}$ such that the following conditions hold:
\begin{enumerate}
    \item [\rm (i)] $R=\bigoplus\limits_{n\in\mathbb{Z}}R_n$.
    \item [\rm (ii)] For every $m,n\in\mathbb{Z}$, $R_mR_n\subseteq \bigoplus\limits_{k\leq m+n} R_k$. 
    \item [\rm (iii)] $1\in R_0$.
\end{enumerate}
\end{definition}
The collection $\{R_n\}_{n\in\mathbb{Z}}$ is called {\em a semi-graduation of} $R$, and we say that the elements of $R_n$ are {\em homogeneous of degree} $n$. $R$ is {\em positively semi-graded} if $R_n=0$ for every $n<0$. If $R$ and $S$ are semi-graded rings and $f: R\rightarrow S$ is a ring homomorphism, then $f$ is called {\em homogeneous} if $f(R_n)\subseteq S_n$ for every $n\in\mathbb{Z}$. 

Note that $\mathbb{N}$-graded rings are positively $\mathsf{SG}$. Ring-theoretical, algebraic and geometric properties of semi-graded rings have been investigated by some mathematicians (e.g., \cite{AbdiTalebi2024, Artamonov2015, Bavula2023, Hamidizadehetal2020, Hashemietal2017, HigueraReyes2023, LouzariReyes2020, ReyesSuarez2020, Seiler2010, SuarezChaconReyes2022, SuarezReyesSuarez2023, Tumwesigyeetal2020}, and references therein).

\begin{definition}[{\cite[Definition 2.2]{LezamaLatorre2017}}]
    Let $R$ be an $\mathsf{SG}$ ring and let $M$ be an $R$-module. We say that $M$ is  {\em semi-graded} if there exists a collection $\{M_n\}_{n\in\mathbb{Z}}$ of subgroups $M_n$ of the additive group $M^+$ such that the following conditions hold:
    \begin{enumerate}
        \item [\rm (i)] $M=\bigoplus\limits_{n\in\mathbb{Z}}M_n$.
        \item [\rm (ii)] For every $m\geq 0$ and $n\in\mathbb{Z}$, $R_mM_n\subseteq \bigoplus\limits_{k\leq m+n}M_k$.
    \end{enumerate}
\end{definition}

The collection $\{M_n\}_{n\in\mathbb{Z}}$ is called {\em a semi-graduation of} $M$, and we say that the elements of $M_n$ are {\em homogeneous of degree} $n$. $M$ is said to be {\em positively semi-graded} if $M_n=0$ for every $n<0$. Let $f: M\rightarrow N$ be a homomorphism of $R$-modules, where $M$ and $N$ are semi-graded $R$-modules, if $f(M_n)\subseteq N_n$ for every $n\in\mathbb{Z}$, then $f$ is called {\em homogeneous}. Let $R$ be an $\mathsf{SG}$ ring, $M$ an $\mathsf{SG}$ $R$-module, and $N$ a submodule of $M$, we say that $N$ is a {\em semi-graded} ($\mathsf{SG}$) {\em submodule of} $M$ if $N=\bigoplus\limits_{n\in\mathbb{Z}}N_n$ where $N_n=M_n\cap N$. In this case, $N$ is an SG $R$-module \cite[Definition 2.3]{LezamaLatorre2017}. 

Different properties of modules over families of semi-graded rings have been investigated in \cite{Fajardoetal2020, Lezama2020, LouzariReyesSpringer2020, NinoRamirezReyes2020, NinoReyes2023, Reyes2019}.

From \cite[Proposition 2.6]{LezamaLatorre2017}, we know that if $R$ is an $\mathsf{SG}$ ring, $M$ is an $\mathsf{SG}$ $R$-module and $N$ is a submodule of $M$, then the following conditions are equivalent:
\begin{enumerate}
        \item [\rm (1)] $N$ is a semi-graded submodule of $M$.
        \item [\rm (2)] For every $z\in N$, the homogeneous components of $z$ are in $N$.
        \item [\rm (3)] $M/N$ is an $\mathsf{SG}$ $R$-module with semi-graduation given by 
        \[
        (M/N)_n=(M_n+N)/N,\ n\in\mathbb{Z}.
        \]
\end{enumerate}

If $M$ is an $\mathsf{SG}$ $R$-module and $\{N_i\}_{i\in I}$ is a family of $\mathsf{SG}$ submodules of $M$, then it is clear that $\bigcap\limits_{i\in I} N_i$ is an $\mathsf{SG}$ submodule of $M$.

Let $X$ be a subset of $M$. We define the $\mathsf{SG}$ submodule generated by $X$ as the intersection of all $\mathsf{SG}$  submodules containing $X$, and we will denote it as $\langle X\rangle^{\mathsf{SG}}$. 

In a similar way, if $R$ is a positively $\mathsf{SG}$ ring, for $t\in \mathbb{N}$, we define $R_{\ge t}$ as the intersection of all two-sided ideals that are $\mathsf{SG}$ submodules containing $\bigoplus\limits_{k\ge t}R_k$. 

Consider an element $n\in \mathbb{Z}$ and the following sets \cite[Section 3.1]{ChaconReyes2022}:
\begin{align*}
R'_n := &\ \{r\in R_n\mid \ {\rm for\ all}\ m\in\mathbb{Z},\ {\rm and}\ {\rm for\ all}\ h\in R_m, rh\in R_{n+m}\},\\
R''_n := &\ \{r\in R'_n\mid \ {\rm for\ all}\ m\in\mathbb{Z},\ {\rm and}\ {\rm for\ all}\ h\in R_m, hr\in R_{n+m}\},\\
R' := &\ \bigcup_{n\in\mathbb{Z}}R'_n,\\
R'' := &\ \bigcup_{n\in\mathbb{Z}}R''_n.
\end{align*}

For $R$ an $\mathsf{SG}$ ring and a left Ore set $S$ of $R$, the first and third authors \cite[Definition 3.8]{ChaconReyes2022} called $S$ {\em good} if the following conditions hold:
    \begin{enumerate}
        \item [\rm (i)] $S\subseteq R''$, and 
        \item [\rm (ii)] if $s\in S$ and $r\in R'$, then there exist elements $u\in R'$ and $v\in S$ such that $us=vr$.
    \end{enumerate}

If $R$ is an $\mathsf{SG}$ ring and $M$ is an $\mathsf{SG}$ $R$-module, then $M$ is called {\em localizable semi-graded} ($\mathsf{LSG}$) if for every element $(n,m)\in \mathbb{Z}^2$, the inclusion $R'_nM_m\subseteq M_{n+m}$ holds \cite[Definition 3.9]{ChaconReyes2022}. The importance of good left Ore sets can be appreciated in the following remark:
\begin{remark}[{\cite[Propositions 3.10, 3.11 and 3.12]{ChaconReyes2022}}]
    Let $R$ be an $\mathsf{SG}$ ring and $S$ a good left Ore set of $R$. Then:
    \begin{enumerate}
        \item [\rm (i)] If $M$ is an $\mathsf{LSG}$ $R$-module, then $S^{-1}M$ is an $\mathsf{LSG}$ $R$-module with semi-graduation given by
\[
    (S^{-1}M)_n = \left\{\frac{f}{s}\mid f\in \bigcup_{k\in\mathbb{Z}} M_k,\ {\rm deg}(f) - {\rm deg}(s)=n \right\}.
\]
\item [\rm (ii)] $S^{-1}R$ is an $\mathsf{SG}$ ring with semigraduation given by
    \[
    (S^{-1}R)_n=\left\{\frac{f}{s}\mid f\in \bigcup_{k\in\mathbb{Z}} R_k,\ {\rm deg}(f) - {\rm deg}(s)=n \right\}.
    \]
\item [\rm (iii)] If $M$ is an $\mathsf{LSG}$ $R$-module, then $S^{-1}M$ is an $\mathsf{SG}$ $S^{-1}R$-module.
    \end{enumerate}
\end{remark}

Having in mind the algebraic properties mentioned above, let $\mathsf{SGR}$ be the {\em category of semi-graded rings} whose objects are the semi-graded rings and whose morphisms are the homogeneous ring homomorphisms. If we fix an $\mathsf{SG}$ ring $R$, $\mathsf{SGR}-R$ denotes the {\em category of semi-graded modules over} $R$ whose morphisms are the homogeneous $R-$homomorphisms. It is straightforward to see that $\mathsf{SGR}-R$ is preadditive, and that its zero object is the trivial module. Let $f:M\rightarrow N$ be a morphism in $\mathsf{SGR}-R$. Since ${\rm Ker}(f)$ and ${\rm Im}(f)$ are semi-graded submodules, it follows that $N/{\rm Im}(f)$ is a semi-graded module. This fact guarantees that the category $\mathsf{SGR}-R$ has kernels and cokernels. If $f$ is a monomorphism of $\mathsf{SGR}-R$, then $f$ is the kernel of the canonical homomorphism $j:N\rightarrow N/{\rm Im}(f)$. If $f$ is an epimorphism, then $f$ is the cokernel of the inclusion $i:{\rm Ker}(f)\rightarrow M$. In this way, the category $\mathsf{SGR}-R$ is normal and conormal.

If $\{M_i\}_{i\in I}$ is a family of objects of $\mathsf{SGR}-R$, then their direct sum $\bigoplus\limits_{i\in I} M_i$ is a semi-graded ring with semi-graduation given by
\[
\biggl(\bigoplus_{i\in I} M_i\biggr)_p := \bigoplus_{i\in I}(M_i)_p, \ p\in\mathbb{Z}.
\]

It is easy to see that this object with the natural inclusions coincides with the coproduct of the familiy of objects $\{M_i\}_{i\in I}$ in $\mathsf{SGR}-R$. Therefore, $\mathsf{SGR}-R$ is an Abelian category. Finally, let $\mathsf{LSG}-R$ be the full subcategory of $\mathsf{SGR}-R$ whose objects are the  $\mathsf{LSG}$ $R$-modules. This subcategory is closed for subobjects, quotients and coproducts, so it is Abelian (see \cite[Section 3.2]{ChaconReyes2022} for more details).

Let $R$ be an $\mathsf{SG}$-ring and $M$ an $R$-module. An element $m\in M$ is a {\em torsion element} if there exist $n, t \ge 0$ such that $R_{\ge t}^n m = 0$. The set of torsion elements of $M$ is denoted by $T(M)$. $M$ is called a {\em torsion module} if $T(M)=M$, and it is said to be {\em torsion-free} if $T(M)=0$. From \cite[Remark 5.4]{LezamaLatorre2017} we know that $T(M)$ is an $R$-submodule of $M$.

The following notion of {\em schematicness} introduced by Chac\'on and Reyes \cite{ChaconPhD2022, ChaconReyes2022} to the class of semi-graded rings extends the corresponding defined by Van Oystaeyen and Willaert \cite{VanOystaeyenWillaert1995} in the $\mathbb{N}$-graded case. Before, for a positively $\mathsf{SG}$ ring $R$, if $R_+ :=\bigoplus\limits_{k\ge 1} R_k$, then we say that a left Ore set $S$ is {\em non-trivial} if $S\cap R_+\neq\emptyset$.

\begin{definition}[{\cite[Definition 4.1]{ChaconReyes2022}}]\label{def.schematic}
Let $R$ be a positively $\mathsf{SG}$ left Noetherian ring. $R$ is called ({\em left}) {\em schematic} if there is a finite set $I$ of non-trivial good left Ore sets of $R$ such that for each $(x_S)_{S\in I}\in \prod\limits_{S\in I}S$, there exist $t,m\in \mathbb{N}$ such that $(R_{\ge t})^m\subseteq\sum\limits_{S\in I}Rx_s$.
\end{definition}

If $R$ is schematic by considering the {\em good} left Ore sets $S_i$, and defining for each $i$,
\[
\kappa_{S_i}(M)=\{m\in M\mid \exists s\in S_i, sm=0 \}
\]
then $\bigcap\limits_{i=1}^n\kappa_{S_i}(M) = T(M)$ for every $\mathsf{SG}$ $R$-module $M$. If $M$ is an $\mathsf{LSG}$ $R$-module, then for each $i = 1,\dotsc, n$, $\kappa_{S_i}(M)$ is an $\mathsf{SG}$ submodule of $M$, and so $T(M)$ is also an $\mathsf{SG}$ submodule of $M$. These facts imply that $M/T(M)$ is an $\mathsf{SG}$ $R$-module.

\section{Maps between schematic semi-graded rings}\label{semigradedsettingnew}

This section contains the original results of the paper. We start with the following fact on morphisms in the category $\mathsf{SGR}-R$.

\begin{proposition}\label{PropositionstartingSG}
Let $f:N\rightarrow M$ be a morphism of $\mathsf{SGR}-R$ and $X\subseteq N$. Then $f(\langle X\rangle^{\mathsf{SG}})=\langle f(X)\rangle^{\mathsf{SG}}$.
\end{proposition}
\begin{proof}
Let $y\in \langle f(X)\rangle^{\mathsf{SG}}$. Then $y\in M'$ for every $\mathsf{SG}$-submodule $M'$ of $M$ with $f(X) \subseteq M'$. Since $\langle X\rangle^{\mathsf{SG}}$ is an $\mathsf{SG}$-submodule of $N$ and $X \subseteq \langle X\rangle^{\mathsf{SG}}$, then $f(\langle X\rangle^{\mathsf{SG}})$ is an $\mathsf{SG}$-submodule of $M$ such that $f(X) \subseteq f(\langle X\rangle^{\mathsf{SG}})$, whence $y\in f(\langle X\rangle^{\mathsf{SG}})$. This shows that $\langle f(X)\rangle^{\mathsf{SG}}\subseteq f(\langle X\rangle^{\mathsf{SG}})$.

On the other hand, if $y\in f(\langle X\rangle^{\mathsf{SG}})$ then there exists $x\in \langle X\rangle^{\mathsf{SG}}$ with $f(x)=y$. Let $M'$ be an $\mathsf{SG}$-submodule of $M$ such that $M'\supseteq f(X)$. Then $f^{-1}(M')$ is an $\mathsf{SG}$-submodule of $N$ with $X \subseteq f^{-1}(M')$, which implies that $x\in f^{-1}(M')$, and hence $y = f(x)\in M'$ and $y\in \langle f(X)\rangle^{\mathsf{SG}}$. Thus, $f(\langle X\rangle^{\mathsf{SG}})\subseteq \langle f(X)\rangle^{\mathsf{SG}}$.   
\end{proof}

Let $\mathsf{TOR}-R$ be the full subcategory of $\mathsf{SGR}-R$ consisting of the torsion modules of $R$. Following the ideas presented by Lezama and Latorre \cite[Theorem 5.5]{LezamaLatorre2017}, it is easy to see that $\mathsf{TOR}-R$ is a Serre subcategory of $\mathsf{SGR}-R$.

For $R$ and $S$ $\mathsf{SG}$ rings, if $f:R\rightarrow S$ is a homomorphism of $\mathsf{SG}$ rings and $M$ is an $\mathsf{SG}$ $S$-module, by defining $rm:=f(r)m$, for elements $r\in R_a$ and $m\in M_b$, we get $f(r)\in S_a$, and so $rm = f(r)m \in \bigoplus\limits_{c\le a+b}M_c$. This guarantees that $M$ is also an $\mathsf{SG}$ $R$-module with the semi-graduation given as $S$-module. Under these conditions, it is clear that if $g:M\rightarrow N$ is a homomorphism of $\mathsf{SG}$ $S$-modules, then $g$ is also a homomorphism of $\mathsf{SG}$ $R$-modules. In this way, we obtain the functor 
\begin{align}\label{functorone}
    f_*:\mathsf{SGR-S} &\ \rightarrow \mathsf{SGR-R},\\
    M &\ \mapsto f_{*}(M)=M\\
    g &\ \mapsto f_{*}(g)=g
\end{align}

where $f_*(M)$ is the same module $M$ considered as an $R$-module. In particular, if $J$ is an $\mathsf{SG}$ ideal of $R$ (that is, a two-sided ideal that is an $\mathsf{SG}$ submodule of $R$), then $R/J$ is an $\mathsf{SG}$ ring and the canonical morphism $f:R\rightarrow R/J$ is homogeneous. This fact guarantees the existence of the functor
\begin{equation}\label{functortwo}
 f_*:\mathsf{SGR}-R/J\rightarrow \mathsf{SGR}-R.
\end{equation}

As in the case of Smith \cite[Theorem 3.2]{Smith2003}, \cite[Theorem 1.2]{Smith2016}, we are interested in the existence of induced $\mathsf{closed\  immersions}$ of projective spaces, so our aim in this section is to define a pair of adjoint functors to $f_{*}$.

\begin{definition}
Consider $M$ an $\mathsf{SG}$ $R$-module. We define $f^!(M)$ as the largest $\mathsf{SG}$ $R$-submodule of $M$ that is annihilated by $J$. 
\end{definition}

We have the following immediate result on $f^!(M)$. 

\begin{proposition}\label{Propoprel1}
$f^!(M)$ is an $\mathsf{SG}$ $R/J$-module by defining $\overline{r}m = rm$, for every $\overline{r}\in R/J$ and each $m\in f^!(M)$.
\begin{proof}
Let $r,s\in R$ with $\overline{r}=\overline{s}$. Then $r-s\in J$, and since $m\in f^!(M)$, $m$ is annihilated by $J$, which implies that $0 = (r-s)m = rm-sm$, i.e., $rm = sm$. It is clear that the product defined satisfies the definition of module. Finally, for elements $\overline{r}\in \left(R/J\right)_a$ and $m\in M_b$, without loss of generality we can take $r\in R_a$, and get $\overline{r}m=rm\in \bigoplus_{c\le a+b}M_c$, whence the assertion follows.
\end{proof}
\end{proposition}

\begin{proposition}\label{Propoprel2}
Let $g:M\rightarrow N$ be a morphism of $\mathsf{SG}$ $R$-modules. Then $g(f^!(M))\subseteq f^!(N)$.
\begin{proof}
Consider $m\in f^!(M)$ and $j\in J$. Since $m$ is annihilated by $J$, $jm = 0$, and so $jg(m)=g(jm)=g(0)=0$. This shows that $g(m)$ is annihilated by $J$, and so $g(m)\in f^!(N)$. 
\end{proof}
\end{proposition}

Propositions \ref{Propoprel1} and \ref{Propoprel2} allow us to define the functor 
\begin{align}\label{functorthree}
f^{!}: \mathsf{SGR}-R &\ \rightarrow \mathsf{SGR}-R/J\\
M &\ \mapsto f^{!}(M)\\
g &\ \mapsto f^!(g) = g|_{f^!(M)},
\end{align}

that sends every morphism to its restriction. If no confusion arises, we write $f^!(g) = g$.

\begin{proposition}\label{Adjointparfirst}
$(f_*,f^!)$ is an adjoint pair.
\begin{proof}
Consider $M$ and $N$ be objects belonging to $\mathsf{SGR}-R/J$ and $\mathsf{SGR}-R$, respectively. If $g:f_{*}(M) \rightarrow N$ is a morphism of $\mathsf{SGR}-R$, then it is clear that ${\rm Im}(g)\subseteq f^!(N)$, and hence we get the isomorphism $\mu_{MN}:{\rm Hom}(M,f^!(N))\rightarrow {\rm Hom}(f_*(M),N)$ given by $\mu_{MN}(g)=g$. This guarantees the required natural isomorphism.
\end{proof}
\end{proposition}

With the aim of defining the other adjoint pair, recall that in the graded setting, to every $R$-module $M$ is assigned the $R/J$-module $M/JM$. However, since in the case of semi-graded rings $JM$ is not necessarily an $\mathsf{SG}$ submodule of $M$ (for instance, take $R = \Bbbk[x]$, $J = Rx$, and $M = A_1(\Bbbk)$, the first Weyl algebra), then we cannot assert that $M/JM$ is an object in $\mathsf{SGR}-R/J$. This fact motivates us to consider $f^*(M) := M/\langle JM\rangle ^{\mathsf{SG}}$. Precisely, since $\langle JM\rangle ^{\mathsf{SG}}$ is an $\mathsf{SG}$ submodule of $M$, then $f^*(M)$ is also an $\mathsf{SG}$ $R$-module, and it is easy to see that $f^*(M)$ is an $\mathsf{SG}$ $R/J$-module by defining $\overline{a}\cdot\overline{m} = \overline{am}$, and considering the semi-graduation of $f^*(M)$ as an $\mathsf{SG}$ $R$-module. In this way, if $\alpha:N\rightarrow M$ is a morphism of $\mathsf{SGR}-R$ modules, then it is possible to define 
\begin{align*}
\alpha^*:f^*(N)&\ \rightarrow f^*(M) \\
\overline{m} &\ \mapsto \overline{\alpha(m)}.
\end{align*}

More exactly,

\begin{proposition}
    $\alpha^*$ is a morphism of the category $\mathsf{SGR}-R/J$.
\begin{proof}
Since $\alpha(JN)\subseteq JM$, then $\alpha(\langle JN\rangle ^{\mathsf{SG}}) = \langle \alpha(JN)\rangle ^{\mathsf{SG}}\subseteq \langle JM\rangle ^{\mathsf{SG}}$, which shows that $\alpha^*$ is well-defined. It is immediate to see that $\alpha^*$ is a homogeneous $R/J$-homomorphism.
\end{proof}
\end{proposition}

From the discussion above, we assert the existence of the functor 
\begin{equation}\label{functorfour}
f^*:\mathsf{SGR}-R\rightarrow \mathsf{SGR}-R/J,
\end{equation}

such that the following proposition holds:
\begin{proposition}\label{Adjointparsecond}
$(f^*,f_*)$ is an adjoint pair.
\begin{proof}
Consider $M$ an object in $\mathsf{SGR}-R$ and $N$ an object in $\mathsf{SGR}-R/J$. Let $h : M\rightarrow f_*(N)$ a morphism in $\mathsf{SGR}-R$. If $m\in M$ and $j\in J$, then $h(jm) = jh(m) = \overline{j}h(m) = 0$, which implies $JM\subseteq {\rm Ker}(h)$, and hence $\langle JM\rangle^{\mathsf{SG}}\subseteq {\rm Ker}(h)$. Then the map $\lambda_{MN}(h):f^*(M)\rightarrow N$ given by $\lambda_{MN}(h)(\overline{m})=h(m)$ is well-defined, and in fact, it is clear that it is a morphism in $\mathsf{SGR}-R/J$.

On the other hand, if $g:f^*(M)\rightarrow N$ is a morphism in $\mathsf{SGR}-R/J$, it is straightforward to see that the map $\lambda'_{MN}(g):M\rightarrow f_*(N)$ defined by $\lambda'_{MN}(g)(m)=g(\overline{m})$ is a morphism in $\mathsf{SGR}-R$.

Finally, note that the maps 
\begin{align*}
\lambda_{MN} : &\ {\rm Hom}(M,f_*(N))\rightarrow {\rm Hom}(f^*(M),N), \ {\rm and} \\
\lambda'_{MN} : &\ {\rm Hom}(f^*(M),N)\rightarrow {\rm Hom}(M,f_*(N))
\end{align*}

are inverses of each other, and correspond to the natural isomorphisms that guarantee that $(f^*,f_*)$ is an adjoint pair.
\end{proof}
\end{proposition}

\begin{definition}
Let $\mathcal{A}$ be a subcategory of $\mathsf{SGR}-R$ (or $\mathsf{Mod}-R$). We say that an object $E$ of $\mathcal{A}$ is $T$-{\em injective} if the following condition holds: for an object $M$ and $N$ a subobject of $M$ (both in $\mathcal{A}$) with $M/N$ being of $T$-torsion, then every morphism from $N$ to $E$ can be extended to $M$. If this extension is unique, then $E$ is said to be $T$-closed. 
\end{definition}

As it occurs in the graded case, $E$ is $T$-closed if and only if $E$ is $T$-torsionfree and $T$-injective. Note that this notion of being $T$-closed is equivalent to the corresponding introduced by Gabriel \cite{Gabriel1962}. 

If $M$ is an $R$-module, there exists a $T$-closed module $Q(M)$ in $\mathsf{Mod}-R$ which is called the {\it module of quotients} of $M$ satisfying the following property: there exists an injective homomorphism $\Phi_M:M/T(M)\rightarrow Q(M)$ such that ${\rm Coker}(\Phi_M)$ is torsion. For more details on the construction of $Q(M)$, see Goldman \cite[Section 3]{Goldman}, Stenstrom \cite[Chapter IX]{Stenstrom1975} or Van Oystaeyen \cite{VanOystaeyen1978}. 

If $M$ is an object in $\mathsf{LSG}-R$, then $Q(M)$ is also an object in $\mathsf{LSG}-R$ \cite[Theorem 4.11]{ChaconReyes2022}. In this case, the semi-graduation of $Q(M)$ is defined as follows: since $M'={\rm Im}(\Phi_M)$ is isomorphic to $M/T(M)$, then $M'$ has a natural semi-graduation. Now, for $\xi\in Q(M)$, we say that $\xi$ is {\em homogeneous of degree} $k$ if and only if there exist elements $n,t\in\mathbb{N}$ such that $R_{\ge n}^t\xi\subseteq M'$, and for all $s\in R_{\ge t}^n\cap R'$, the element $s\xi$ is homogeneous of degree $k+{\deg}(s)$. By considering this semi-graduation, it is straightforward to see that $\Phi_M$ is a morphism in $\mathsf{LSG}-R$.
 
\begin{proposition}
If $M$ is an $\mathsf{LSG}-R$ module, then $Q(M)$ is $T$-closed in $\mathsf{LSG}-R$.
    \begin{proof}
We know that $Q(M)$ is $T$-torsionfree and $T$-injective in $\mathsf{Mod}-R$. Let $N\in \mathsf{LSG}-R$, $N'$ an $\mathsf{SG}$ submodule of $N$ such that $N/N'$ is $T$-torsion, and $f:N'\rightarrow Q(M)$ an homogeneous $R$-homomorphism. Since $Q(M)$ is $T$-injective in $\mathsf{Mod}-R$, there exists an $R$-homomorphism $g:N\rightarrow Q(M)$ that extends $f$, we only have to show that $g$ is also homogeneous.

Consider $m\in N_k$. By using that $N/N'$ is $T$-torsion, there exist elements $n_1,t_1 \in \mathbb{N}$ with $R_{\ge n_1}^{t_1}m\subseteq N'$. On the other hand, let $M'$ be the isomorphic image of $M/T(M)$ in $Q(M)$. Since $Q(M)/M'$ is $T$-torsion, there exist $n_2,t_2\in\mathbb{N}$ with $R_{\ge n_2}^{t_2}g(m)\subseteq M'$. If $n := {\rm max}\{n_1,n_2\}$ and $t := {\rm max}\{t_1,t_2\}$, then $R_{\ge n}^tg(m)\subseteq M'$, and for every $s\in R_{\ge n}^t\cap R'$, we get that $sm$ is homogeneous and $sm\in N'$ (recall that $N$ is $\mathsf{LSG}$). Hence, $sg(m) = g(sm) = f(sm)\in (M')_{{\rm deg}(s)+k}$ due to that $f$ is homogeneous. By considering the semi-graduation of $Q(M)$, it follows that $g(m)\in (Q(M))_k$, whence $g$ is homogeneous.
\end{proof}
\end{proposition}

Since in general we do not know if the canonical functor $\pi:\mathsf{SGR}-R\rightarrow \mathsf{SGR}-R/\mathsf{TOR}-R$ has a right adjoint, from now on we consider a schematic ring $R$ and the full subcategory $\mathsf{LSG}-R$. 

Since $\mathsf{LSG}-R$ is a subcategory of $\mathsf{SGR}-R$ closed for subobjects and quotients, then $\tau := \mathsf{TOR}-R \cap \mathsf{LSG}-R$ is a Serre subcategory of $\mathsf{LSG}-R$. As a matter of fact, due to \cite[Proposition 4, Section 17]{Gabriel1962}, the canonical functor $\pi:\mathsf{LSG}-R\rightarrow \mathsf{LSG}-R/\tau$ has a right adjoint, say $\omega$. On the other hand, by \cite[Corollary 4, Section 17; Proposition 13(a), Section 17]{Gabriel1962} the category $\mathsf{LSG}-R/\tau$ is equivalent to the full subcategory $\mathcal{C}-R$ of $\mathsf{LSG}-R$ consisting of the $T$-closed modules. Note that we can define the functor
\begin{align*}
  Q:\mathsf{LSG}-R &\ \rightarrow \mathcal{C}-R \\
  M &\ \mapsto Q(M),
\end{align*}

and in fact, $Q = \omega\pi$. Without loss of generality, from now on we assume that $Q(M)$ is an extension of $M/T(M)$ and $\Phi_M$ is the inclusion map.

Having in mind that our objects of interest are schematic rings, we need to guarantee that if $J$ is a $\mathsf{SG}$ ideal of a schematic ring $R$ then $R/J$ is also schematic.  We say that $J$ is a {\em compatible} ideal of $R$ if $R'/J= \left(R/J\right)'$. The following result shows the importance of this condition.

\begin{proposition}
If $R$ is a schematic ring and $J$ is a compatible $\mathsf{SG}$ ideal of $R$, then $R/J$ is also schematic.
\begin{proof}
It is easy to see that if $S$ is a good Ore set of $R$, then $S/J$ is also a multiplicative set that satisfies Ore's condition. Now, by using that $R$ is left Noetherian, then $R/J$ also is, and it is well-known that $S/J$ is an Ore set of $R/J$. Since $J$ is compatible, we get $S/J$ is a good Ore set.

On the other hand, for every $n\in\mathbb{N}$, it can be seen that $R_{\ge n}/J=\left(R/J\right)_{\ge n}$, which guarantees that for $S_1, \dotsc, S_k$ good Ore sets of $R$ satisfying the schematic condition, then $S_1/J,\dots,S_k/J$ are good Ore sets of $R/J$ satisfying this condition. 
    \end{proof}
\end{proposition}

Notice that if $M$ is an $\mathsf{LSG}-R/J$ module, then $f_*(M)$ is an $\mathsf{LSG}-R$ module. If $J$ is a compatible ideal of $R$, then for each $\mathsf{LSG}-R$ module $M$, we get that $f^*(M)$ and $f^!(M)$ are $\mathsf{LSG}-R/J$ modules.

\begin{proposition}
    If $E$ is an object in $\mathcal{C}-R$, then $f^!(E)$ is an object in $\mathcal{C}-R/J$.
\end{proposition}
\begin{proof}
Consider $M$ an object in $\mathsf{LSG}-R/J$, $N$ a subobject of $M$ with $M/N$ of torsion, and $g:N\rightarrow f^!(E)$ a morphism in $\mathsf{LSG}-R/J$. Let $j:f^!(E)\rightarrow E$ be the inclusion. By considering the morphism $j\circ g:f_*(N)\rightarrow E$ and using that $E$ is closed, there exists a homogeneous $R$-homomorphism $h:f_*(M)\rightarrow E$ that extends to $j\circ g$. Besides, it is clear that $f^!(h):M\rightarrow f^!(E)$ extends to $g$. 

Finally, due to $f^!(E)$ is a submodule of $E$, and $E$ is torsion-free, then $f^!(E)$ is torsion-free in $\mathsf{LSG}-R$, and so it is torsion-free in $\mathsf{LSG}-R/J$. 
\end{proof}

We define the morphisms
\[
i_*:\mathcal{C}-R/J\rightarrow \mathcal{C}-R,\quad i^!:\mathcal{C}-R\rightarrow \mathcal{C}-R/J,\quad {\rm and}\quad i^*:\mathcal{C}-R\rightarrow \mathcal{C}-R/J,
\]

as the restriction maps of $Qf_*$, $f^!$, and $Q'f^*$, respectively, where $Q$ is the functor assigning the module of quotients in the category $\mathsf{LSG}-R$ and $Q'$ is the same functor in $\mathsf{LSG}-R/J$.

\begin{lemma}\label{Prop adjoint_and_fully}
 $(i_*,i^!)$ is an adjoint pair and $i_*$ is fully faithful.
\end{lemma}
\begin{proof}
    \begin{enumerate}
        \item $(i_*,i^!)$ is an adjoint pair: Let $M$ be an object in $\mathcal{C}-R/J$ and $N$ an object in $\mathcal{C}-R$. If $g:M\rightarrow f^!(N)$ is a homogenous $R/J$-homomorphism and $j: f^!(N)\rightarrow N$ is the inclusion map, then $j\circ g:M\rightarrow N$ is a homogeneous $R$-homomorphism. Since $N$ is $T$-closed, there exists a unique map $\overline{g}:i_*(M)=Q(M)\rightarrow N$ extending $j\circ g$. This fact allows to obtain the map $\gamma_{MN}:{\rm Hom}(M,f^!(N))\rightarrow {\rm Hom}(i_*(M),N)$ defined by $\gamma_{MN}(g)=\overline{g}$. It is clear that $\gamma_{MN}$ is a monomorphism. 

On the other hand, note that $M\subseteq f^!(Q(M))$ and $f^!(Q(M))/M$ is a torsion module, whence $f^!(Q(M))$ is the module of quotients of $M$ (in $\mathsf{LSG}-R/J$), and by using that $M$ is $T$-closed, we get $f^!(Q(M))=M$. In this way, if $h:i_*(M)\rightarrow N$ is a homogeneous $R$-homomorphism, then $h = \gamma_{MN}(f^!(h))$, and so $\gamma_{MN}$ is a bijective map. It is straightforward to see that these maps induce the natural isomorphism that show that $(i_*,i^!)$ is an adjoint pair.
    \item $i_*$ is fully faithful: From the definition of $f_*$ we know that it is fully faithful. It is easy to show that when restricted to free torsion modules, the functor $Q$ is also fully faithful. From these facts we get that $i_*$ is fully faithful
    \end{enumerate}
\end{proof}

With the aim of proving the other adjunction, we require that $J$ satisfies the following additional ($\ast$): for any elements $n, t \in \mathbb{N}$, when $x\in J\cap R_{\ge t}^n$ there exist $n_x,t_x\in\mathbb{N}$ such that $R_{\ge t_x}^{n_x}x\subseteq JR_{\ge t}^n$. Let us just mention that this condition holds trivially in commutative cases. The importance of condition ($\ast$) is shown in the following lemma which is key in the proof of Theorem \ref{Smith2003Theorem3.2Generalization}.

\begin{lemma}\label{lemma}
If $J$ satisfies condition {\rm (}$\ast${\rm )} and $E$ is an object in $\mathcal{C}-R/J$, then $f_*(E)$ is an object in $\mathcal{C}-R$. 
\end{lemma}
\begin{proof}
Recall that the construction of the module of quotients in $\mathsf{LSG}-R$ is the same as in the category $\mathsf{Mod}-R$, so it follows that a semi-graded $R$-module $M$ is closed in $\mathsf{LSG}-R$ if and only if is closed in $\mathsf{Mod}-R$. In this way, we only have to show that $E$ is closed in $\mathsf{Mod}-R$.

Consider elements $n, t \in\mathbb{N},\ I = R_{\ge t}^n$ and $g:I\rightarrow E$ an $R$-homomorphism. We define the maps 
\begin{align*}
    g^* &\ :I/JI\rightarrow E, \quad \overline{x} \mapsto g(x) \\ 
    j &\ :I/JI\rightarrow R/J, \quad \overline{x} \mapsto x.
\end{align*}

Since $E$ is an $R/J$-module, $g^*$ is well-defined. It is clear that $g$ and $j$ are $R/J$-homomorphisms. Note that ${\rm Coker}(j)$ is torsion, and by using that $J$ satisfies ($\ast$), we get that ${\rm Ker}(j)$ is torsion. As $E$ is closed in $\mathsf{Mod}-R/J$ there exist an $R/J$-homomorphism $h : R/J\rightarrow E$ with $h\circ j = g^*$. If we take the canonical map $\theta: R\rightarrow R/J$, then the $R$-homomorphism $h\circ \theta:R\rightarrow E$ extends $g$, and by \cite[Proposition 3.2]{Goldman} $E$ is $T$-injective in $\mathsf{Mod}-R$. Besides, since $E$ is $T$-torsionfree in $\mathsf{Mod}-R$ (due to that $E$ is closed in $\mathsf{Mod}-R/J$), it follows that $E$ is $T$-closed in $\mathsf{Mod}-R$.
\end{proof}

Finally, we get the most important result of the paper that extends partially Smith's result \cite[Theorem 3.2]{Smith2003}, \cite[Theorem 1.2]{Smith2016}.

\begin{theorem}\label{Smith2003Theorem3.2Generalization}
Let $R$ be a schematic $\mathsf{SG}$ ring and $J$ a compatible $\mathsf{SG}$ ideal of $R$ satisfying {\rm (}$\ast${\rm )}. Then the map $i:\mathcal{C}-R/J\rightarrow \mathcal{C}-R$ given by the functors $(i^*,i_*,i^!)$ is a closed immersion.
\end{theorem}
\begin{proof}
By Lemma \ref{Prop adjoint_and_fully}, we only have to show that $(i^*,i_*)$ is an adjoint pair and that the essential image of $i_*$ is closed for subobjects and quotients.

Let $M$ be an object in $\mathcal{C}-R$ and $N$ an object in $\mathcal{C}-R/J$. Due to Lemma \ref{lemma}, we get $(i_*=f_*)$, and hence the morphism $v_{MN}:{\rm Hom}(M,i_*(N))\rightarrow {\rm Hom}(i^*(M),N)$ defined by $v_{MN}=Q'\circ\lambda_{MN}$, where $\lambda_{MN}$ is the map defined in \ref{Adjointparsecond}, induces the required natural isomorphism.

By last, note that in the category $\mathcal{C}-R$ the subobjects of an object $M$ are the closed $\mathsf{SG}$-submodules of $M$, and if $N$ is a subobject of $M$ then the quotient is given by $Q(M/N)$. Besides, since ${\rm Im}(i_*)$ is the subcategory consisting of the $\mathsf{SG}$-$R$-modules $M$ such that $f^!(M)=M$, and in this case we have $Q(M/N)=Q'(M/N)$, it follows that ${\rm Im}(i_*)$ is closed under quotients  and submodules.
\end{proof}

We get immediately Smith's result in the schematic graded case formulated by Van Oystaeyen and Willaert \cite{VanOystaeyenWillaert1995}.

\begin{corollary}[{\cite[Theorem 3.2]{Smith2003}; \cite[Theorem 1.2]{Smith2016}}]\label{Smith2003Theorem3.2}
Let $J$ be a graded ideal in an $\mathbb{N}$-graded schematic $\Bbbk$-algebra $A$ satisfying ($\ast$). Then the homomorphism $A \to A / J$ induces a closed immersion $i: {\rm Proj}_{nc} A / J \to {\rm Proj}_{nc} A$.
\end{corollary}
\begin{proof}
Since $A$ is $\mathbb{N}$-graded, then the category $\mathsf{GrMod}-A$ coincides with the category $\mathsf{LSG}-A$. Finally, it is clear that in the graded case, the graded ideals are $\mathsf{SG}$ compatible.
\end{proof}

\section{Acknowledgments}

We would like to express our sincere thanks to Professor Jason Gaddis for introducing us to this project.

\section{Conclusions and future work}\label{conclusionsfuturework}

For $R$ a schematic $\mathsf{SG}$ ring and $J$ a compatible $\mathsf{SG}$ ideal of $R$ that satisfies ($\ast$), we have showed that the canonical map $R\rightarrow R/J$ induces a categorical map $\mathcal{C}-R/J\rightarrow \mathcal{C}-R$. Since in the $\mathbb{N}$-graded setting the categories $\mathsf{LSG}-R$ and $\mathsf{gr} R$ coincide, it follows that the category $\mathcal{C}-R$ is equivalent to the category ${\rm Proj}_{nc}(R)$ defined by Smith, Theorem \ref{Smith2003Theorem3.2Generalization} presents a partial generalization of Smith's result \cite[Theorem 3.2]{Smith2003}; \cite[Theorem 1.2]{Smith2016} since he did not assume the condition of schematicness on the $\mathbb{N}$-graded ring. As we saw above, in the semi-graded case we need this condition to guarantee the existence of the module of quotients in (at least) the category $\mathsf{LSG}-R$. An immediate task is to investigate if the category $\mathsf{SGR}-R$, or an appropriate subcategory of this, has modules of quotients. Thinking about it, a possible way to solve this problem is due to Gabriel \cite{Gabriel1962}, which consists in to prove that $\mathsf{SGR}-R$ has enough injectives (in fact, this would allow to formulate of notion of {\em cohomology} in the semi-graded setting). Precisely, it is important to note that our methodology here is very different from that considered by Smith, since he used the well-known fact that the category of modules over an $\mathbb{N}$-graded ring is a Grothendieck category, while in the semi-graded context we do not know if $\mathsf{SGR}-R$ is also Grothendieck; we only know that $\mathsf{SGR}-R$ is $\texttt{Ab}5$ \cite[Section 1.5]{ChaconPhD2022}. Of course, if we found a positive answer then automatically we guarantee the existence of enough injectives in $\mathsf{SGR}-R$. This will be our line of research in the near future.

Last but not least, we are also interested in Smith's papers \cite{Smith2001, Smith2002}. In the first, he called a non-commutative space $X$ {\em integral} if there is an indecomposable injective $X$-module $\mathscr{E}_X$ such that its endomorphism ring is a division ring and every $X$-module is a subquotient of a direct sum of copies of $\mathscr{E}_X$. A Noetherian scheme is integral in this sense if and only if it is integral in the usual sense. Smith proved that several classes of non-commutative spaces over $\mathbb{N}$-graded rings are integral. On the other hand, in \cite{Smith2002} he investigated the concepts of closed points, closed subspaces, open subspaces, weakly closed and weakly open subspaces, and effective divisors, on a non-commutative space over an $\mathbb{N}$-graded ring. As expected, the characterization of all these notions in the setting of (schematic) semi-graded rings is a very interesting problem.

\end{document}